\documentclass{aims}

\usepackage{amsmath,amsfonts,amsthm,amssymb,amsxtra,times}
\usepackage{epsfig}
  \usepackage{paralist}
 \usepackage[colorlinks=true]{hyperref}

\def\currentvolume{X}
 \def\currentissue{X}
  \def\currentyear{200X}
   \def\currentmonth{XX}
    \def\ppages{X--XX}

 \newtheorem{theorem}{Theorem}[section]
 \newtheorem{corollary}[theorem]{Corollary}
 \newtheorem{lemma}[theorem]{Lemma}

\theoremstyle{definition}
\newtheorem{rem}[theorem]{Remark}

\newcommand{\R}{\mathbb R}
\newcommand{\C}{\mathbb C}
\newcommand{\LH}{{\mathcal L}(H) }
\newcommand{\Rat}{{\mathcal R}}

\renewcommand{\Re}{\mathop {\rm Re}\nolimits}
\renewcommand{\Im}{\mathop {\rm Im}\nolimits}

\title[Intersections of disks]{Intersections of several disks of the Riemann sphere as $K$-spectral sets}

\author{Catalin Badea}
\address{Laboratoire Paul Painlev\' e, B\^at. M2, UMR CNRS no. 8524,
Universit\'e de Lille 1, 59655 Villeneuve d'Ascq Cedex, France}
\email{Catalin.Badea@math.univ-lille1.fr}
\thanks{The first author was partially supported by ANR-Projet Blanc DYNOP. The third author was partially supported by GDR 2753 CNRS}

\author{Bernhard Beckermann}
\address{Laboratoire Paul Painlev\' e, B\^at. M2, UMR CNRS no. 8524,
Universit\'e de Lille 1, 59655 Villeneuve d'Ascq Cedex, France}
\email{Bernhard.Beckermann@math.univ-lille1.fr}
\author{Michel Crouzeix}
\address{Institut de Recherche Math\'ematique de Rennes, UMR no. 6625,
Universit\'e de Rennes 1, Campus de Beaulieu, 35042
RENNES Cedex, France}
\email{michel.crouzeix@univ-rennes1.fr}

\date{\today}
\dedicatory{Pour Philippe Ciarlet, \`a l'occasion de son soixante-dixi\`eme anniversaire}

\begin{document}

\begin{abstract}

We prove that if $n$ closed disks $D_1$, $D_2$,\dots, $D_n$, of
the Riemann sphere are spectral sets for a bounded linear operator
$A$ on a Hilbert space, then their intersection $D_1\cap
D_2\dots\cap D_n$ is a complete $K$-spectral set for $A$, with
$K\leq n+n(n\!-\!1)/\sqrt3$. When $n=2$ and the intersection
$X_1\cap X_2$ is an annulus, this result gives a positive answer
to a question of A.L.~Shields (1974).
\end{abstract}

\maketitle

\section{Introduction and the statement of the main results.}
Let $X$ be a closed set in the complex plane and let $\Rat (X)$
denote the algebra of bounded rational functions on $X$, viewed as
a subalgebra of $C(\partial X)$ with the supremum norm $$\|f\|_X =
\sup \{ |f(x)| : x \in X\} = \sup \{ |f(x)| : x \in
\partial X\}. $$ Here $\partial X$ denotes the boundary of the set
$X$.

\medskip
\noindent {\bf Spectral and complete spectral sets.}
 Let $A \in \mathcal{L}(H)$
 be a bounded linear operator acting on a complex Hilbert space $H$.
 For a fixed constant $K > 0$, the set $X$ is said to be a
 $K$-\emph{spectral} set for $A$ if the spectrum $\sigma(A)$ of $A$
 is included in $X$ and the inequality $\|f(A)\| \le K\|f\|_X$ holds
 for every $f \in \Rat (X)$. The poles of a rational
 function $f=p/q \in \Rat (X)$ are outside of $X$, and the operator $f(A)$ is
 naturally defined as $f(A) = p(A)q(A)^{-1}$ or, equivalently,
 by the Riesz holomorphic functional
 calculus. The set $X$ is a \emph{spectral} set
 for $A$ if it is a $K$-spectral set with $K=1$. Thus $X$ is spectral
 for $A$ if and only if $\|\rho_A\| \le 1$,
 where $\rho_A : \Rat (X) \mapsto \mathcal{L}(H)$ is the homomorphism
 given by $\rho_A(f) = f(A)$.

 We denote by $M_n(\Rat (X))$ the algebra of $n$ by $n$ matrices with entries from $\Rat (X)$. If we let the $n$ by $n$ matrices have the operator norm that they inherit as linear transformations on the $n$-dimensional Hilbert space $\C^{n}$, then we can endow $M_n(\Rat (X))$ with the norm
$$\|\left(f_{ij}\right)\|_X = \sup \{ \|\left(f_{ij}(x)\right)\| :
x \in X\} = \sup \{ \|\left(f_{ij}(x)\right)\| : x \in \partial
X\}. $$ In a similar fashion we endow $M_n(\mathcal{L}(H))$ with
the norm it inherits by regarding an element $(A_{ij})$ in
$M_n(\mathcal{L}(H))$ as an operator acting on the direct sum of
$n$ copies of $H$. For a fixed constant $K > 0$, the set $X$ is
said to be a \emph{complete} $K$-\emph{spectral} set for $A$ if
$\sigma(A) \subset X$ and the inequality $\|(f_{ij}(A))\| \le
K\|(f_{ij})\|_X$ holds for every matrix $(f_{ij}) \in M_n(\Rat
(X))$ and every $n$. In terms of the complete bounded norm
(\cite{paul}) of the homomorphism $\rho_A$, this means that
$\|\rho_A\|_{cb} \le K$. A \emph{complete spectral} set is a
complete $K$-spectral set with $K=1$. Complete $K$-spectral sets are important in several problems of Operator Theory (see \cite{paul}).

Spectral sets were introduced and studied by J. von Neumann
\cite{Neu} in 1951. In the same paper von Neumann proved that a
closed disk $\{z \in \C : |z-\alpha| \le r\}$ is a spectral set
for $A$ if and only if $\|A -\alpha I \| \le r$. Also,
the closed set $\{z \in \C : |z-\alpha| \ge r\}$ is spectral for
$A \in \mathcal{L}(H)$ if and only if $\|(A -\alpha I)^{-1} \| \le
r^{-1}$, and the half-plane $\{ \Re(z)\geq 0 \}$ is a spectral set
if and only if $\Re\langle A v,v\rangle\geq0$ for all $v\in H$.
Therefore for any closed disk $D$ of the Riemann sphere
(interior/exterior of a disk or a half-plane) it is easy to check
whether $D$ is a spectral set.

We refer to \cite{RiNa} for a treatment of the
von Neumann's theory of spectral sets and to \cite{paulsen} and
the book \cite{paul} for modern surveys of known properties
of $K$-spectral and complete $K$-spectral sets.

\medskip
%\subsection{Statement of the main results and related findings}
\noindent {\bf The problem.}
Let $X$ be the intersection of $n$ disks of the Riemann sphere,
each of them being a spectral set for a given operator $A \in \LH$.
In the present paper we will be
concerned with the question whether $X$ itself is a (complete)
$K$-spectral set for $A$.

The intersection of two spectral sets is not necessarily a
spectral set; counterexamples for the annulus are presented in \cite{williams, misra, paulsen}. However, the same question for $K$-spectral
sets remains open. The counterexamples for spectral sets show that
the same constant cannot be used for the intersection.

Some cases of the $K$-spectral set problem have been solved. If
two $K$-spectral sets have disjoint boundaries, then by a result
of Douglas and Paulsen \cite{dopa} the intersection is a
$K'$-spectral set for some $K'$. The case when the boundaries meet
was considered by Stampfli \cite{stampfliI,stampfliII} and Lewis
\cite{lewis}. In particular, it is proved in \cite{lewis} that the
intersection of a (complete) $K$-spectral set for the bounded
linear operator $A$ with the closure of any open set containing
the spectrum of $A$ is a (complete) $K'$-spectral set for $A$.

\medskip

\noindent {\bf The main result.}
In this paper we will show the following result, which gives a positive answer to a question raised by Michael A. Dritschel (personal communication).
\begin{theorem}\label{thm0} Let $A\in \mathcal L(H)$, and consider the
intersection $X=D_1\cap D_2\cap\dots\cap D_n$ of $n$ disks of the
Riemann sphere $\overline \C$, each of them being spectral for
$A$. Then $X$ is a complete $K$-spectral set for $A$, with a
constant $K\leq n+n(n\!-\!1)/\sqrt3$.
\end{theorem}

Theorem~\ref{thm0} extends previously known results concerning the
intersection of two or more disks in $\C$ to not necessarily
convex or even connected $X$. Note that the case of finitely
connected compact sets has been studied in \cite{dopa, paul},
however, without a uniform control on the constant $K$.

If in Theorem~\ref{thm0} we add the requirement that the disks
$D_j$ and thus $X$ are convex, then
%it follows from \cite{crzx} that
$X$ is a complete $11.08$-spectral set for $A$. Indeed, the fact
that $D_j$ is a spectral set for $A$ implies that the numerical
range $W(A)=\{\langle Ax, x\rangle : \|x\| = 1\}$ is included in
$D_j$, $j=1,...,n$, and, according to \cite{crzx}, the closure of
the numerical range $W(A)$ is a complete $11.08$-spectral set for
$A$.

 Let us have a closer look at the case
$n=2$ of the intersection of two closed disks of the Riemann
sphere. From Theorem~\ref{thm0} we may conclude that the intersection of two disks of the Riemann sphere is a $K$-spectral set, with $K \leq
2+2/\sqrt 3$. This includes the special case of a sector/strip
obtained by the intersection of two half-planes and discussed in
\cite{crde}, and the lens-shaped intersection of two disks
considered already in \cite{becr}; see also \cite{bcd,becr2}. The case of the annulus is new and it permits to answer a question of A. Shields \cite[Question 7]{shields} as described in the next paragraph.

\medskip

\noindent {\bf Shields' question for the annulus.} Let $R>1$. Given an invertible operator $A$ with $\|A\| \le R$ and
$\|A^{-1}\| \le R$, Shields proved in \cite{shields} that the annulus
$X = X(R)= \{z \in \C : R^{-1} \le |z| \le R\}$ is a $K$-spectral set for $A$ with $K = 2+
\sqrt{\left(R^2\!+\!1\right)/\left(R^2\!-\!1\right)}$. This bound is large if $R$ is close to $1$. In this context, Shield raises the question of finding the smallest
constant $K$ (as a function of $R$) such that the above annulus
is $K$-spectral. In particular \cite[Question 7]{shields}, he
asked whether this optimal constant $K$ remains bounded.
With regard to his second question, we are able to give a
positive answer.

\begin{theorem}\label{thm1}
For $R > 1$, consider the annulus $X = X(R)= \{z \in
\C : R^{-1} \le |z| \le R\}$, and denote by $K(R)$ (and
$K_{cb}(R)$, respectively), the smallest constant $C$ such that
$X$ is a $C$-spectral set (and a complete $C$-spectral set,
respectively) for any invertible $A \in \mathcal{L}(H)$ verifying
$\|A\| \le R$ and $\|A^{-1}\| \le R$. Then
\begin{eqnarray*}
\frac43  <  K(R) \le K_{cb}(R) \le
 2 +  \frac{R+1}{\sqrt{R^2+R+1}}\leq 2 + \frac{2}{\sqrt{3}}.
\end{eqnarray*}
\end{theorem}

This statement has already been established in the unpublished
manuscript \cite{bbc}, where one finds also a preliminary version
of Theorem~\ref{thm0} for $n=2$.
 One consequence of Theorem~\ref{thm1}
is that the upper bound for $K$ of Theorem~\ref{thm0} is not sharp;
furthermore, we expect this bound to be pessimistic for large
$n$.

\medskip

\noindent {\bf Normal $\partial X$-dilations.} Based on results due to W.~Arveson and V.~Paulsen (see \cite[Corollary 7.8, Theorem 9.1]{paul}), it is now a standard fact that if $X$ is a $K$-spectral set for $A$, then there is a normal $\partial X$-dilation for an operator $L^{-1}AL$ similar to $A$. Moreover, the best possible constant $K$ is equal to the infimum of all possible similarity constants (condition numbers) $\|L^{-1}\|\cdot\|L\|$.
We say that $B \in \LH$ has a \emph{normal} $\partial X$-\emph{dilation} if there exists a Hilbert space $\mathcal H$
containing $H$ and a normal operator $N$ on $\mathcal H$ with $\sigma(N)\subset \partial X$ so that
$$ f(B) = P_Hf(N)\mid_H$$ for every rational function $f$ with
poles off $X$. Here $P_H$ is the orthogonal projection of $\mathcal{H}$ onto $H$.

We obtain the following consequence:
\begin{corollary}\label{cor1}
Let $A\in \mathcal L(H)$, and consider the
intersection $X=D_1\cap D_2\cap\dots\cap D_n$ of $n$ disks of the
Riemann sphere $\overline \C$, each of them being spectral for
$A$. Then there exists an invertible
operator $L \in \mathcal{L}(H)$ with $\|L\|\cdot\|L^{-1}\|\leq n+n(n\!-\!1)/\sqrt3$,
such that $L^{-1}AL$ has a normal $\partial X$-dilation.
\end{corollary}
If $\|A\| \le R$ and $\|A^{-1}\| \le R$, and $X=X(R)$ is an annulus, we can apply this corollary (with $n=2$) and obtain a $\partial X(R)$-dilation for a similarity. We would also like to mention the
following deep result due to Agler \cite{agler}: if an annulus $X$ is a
spectral set for $A$, then it is a complete spectral set for $A$,
that is, $A$ has a normal $\partial X$-dilation. However, the
analogue of Agler's result
is not true for triply
connected domains (see \cite{AHR,DrMc, Pic}).

\medskip
%\subsection{Strategy of the proof and organization of the paper}
\noindent {\bf Outline of the proof and organization of the paper.}
Let $A\in \mathcal L(H)$, and consider the
intersection $X=D_1\cap D_2\cap\dots\cap D_n$ of $n$ disks of the
Riemann sphere $\overline \C$, each of them being spectral for
$A$.

\medskip

\noindent {\it Convention.} In what follows we will always suppose that for each spectral set $D_j$, the spectrum $\sigma(A)$ of $A$ is included in the interior of $D_j$. The
general case will then follow by a limit argument by slightly
enlarging the disks. Note that each superset of a spectral set is spectral.

\medskip

\noindent {\it Decomposition of the Cauchy kernel.} Consider a disk $D$ among $D_1, \cdots, D_n$, which is centered in $\omega\in \mathbb C$
(if $D$ is a half-plane we take $\omega=\infty$). We chose an
arclength parametrization $s \mapsto \sigma=\sigma(s)\in \partial
D$ of the boundary of $D$ with orientation such that
$\frac{1}{i}\,\frac{d\sigma}{ds}$ is the outward normal to
$D$. Let $A\in \LH$ be a bounded operator with $\sigma(A)\subset$
int$(D)$. For $\sigma\in
\partial D$, we consider the following Poisson kernel
\begin{equation}\label{def_mu}
\mu(\sigma,A,D) = \tfrac{1}{2\pi i}\big(
(\sigma\!-\!A)^{-1}\tfrac{d\sigma}{ds}-(\bar\sigma\!-\!A^*)^{-1}\tfrac{d\bar\sigma}{ds}
-\tfrac{1}{\sigma\!-\!\omega}\tfrac{d\sigma}{ds}\big).
\end{equation}
Notice that in case $\omega=\infty$ of a half-plane, the term involving
$\omega$ on the right-hand side of (\ref{def_mu}) vanishes.

The first important step in the proof is the decomposition of the Cauchy kernel
$$
 \frac{1}{2\pi i}(\sigma-A)^{-1}d\sigma=\mu(\sigma,A,D)\,ds+\nu(\sigma,A,D)\,d\sigma
$$
as the sum of the Poisson kernel and a residual kernel.
\medskip

\noindent {\it Decomposition of $f(A)$.} For a rational function $f\in \Rat(X)$, the above decomposition of the Cauchy kernel leads to a decomposition of $f(A)$ as
\begin{equation}\label{decomposition}
 f(A)=g_{p}(f) +g_{r}(f),
\end{equation}
with $$ g_{p}(f) =\sum_{j=1}^n \int_{X\cap \partial
D_{j}}f(\sigma)\,\mu(\sigma,A,D_{j})\,ds,\quad g_{r}(f)
=\sum_{j=1}^n \int_{X\cap  \partial
D_j}f(\sigma)\,\nu(\sigma,A,D_{j})\,d\sigma. $$ Here $p$ stands
for ''Poisson`` and $r$ for ''residual``. This decomposition
reduces to the one used before in
  \cite{bbc,crde,becr2} in the special case $\omega=\infty$ of a
  half-plane.
 The proof will
consist in showing that the complete bounded norm of the map
$f\mapsto g_{p}(f)$ is bounded by $n$, while the complete bounded
norm of the map $f\mapsto g_{r}(f)$ can be estimated by
$n(n\!-\!1)/\sqrt3$.

\medskip

Two basic lemmas on operator-valued integrals are proved in
Section \ref{s2.2} and, as application, the Poisson term is
estimated. In order to make clearer the basic ideas of our
reasoning for the control of the residual term, we start with the
annulus case. A proof for Theorem~\ref{thm1} (implying
Theorem~\ref{thm0} for an annulus) is presented in Section
\ref{s2.1}. A proof of Theorem~\ref{thm0} for the case $n=2$ is
presented in Section \ref{s2.5}. For the residual term, the
decisive step is the invariance under M\"obius maps
    (also called fractional linear transformations or
    homographic transformations)
    of our representation formula. We also use the fact that the
variable $\overline \sigma$ which appears in $g_r(f)$ can be
expressed in terms of $\sigma$ due to the particular form of
$\partial X$. Subsequently, a new path of integration is used in
order to monitor the complete bounded norm of the residual term.
The new path of integration will be the circle of radius $1$ in
case of the annulus $\{ R^{-1} \leq | z | \leq R \}$, and the
positive real line in case of the sector $\{ |\arg(z)|\leq \theta
\}$ for $\theta\in (0,\pi/2)$. We call these median lines.

Our proof of Theorem~\ref{thm0} for $n>2$, presented in Section
\ref{s3}, is obtained by constructing a Voronoi-like tesselation
of the Riemann sphere based on the reciprocal of the infinitesimal
Carath\'eodory pseudodistance (see \cite{JaPf}),
  sometimes also called infinitesimal
Carath\'eodory metric (see, e.g., \cite{simha}). Here the new
paths of integration (the median lines) are obtained using
suitable edges of this tesselation, which requires some
combinatorial considerations. Then the arguments used for the case
of two disks will allow us to conclude.

\section{Two basic lemmas}\label{s2.2}
The following two basic
lemmas, formulated in a slightly more general setting, are obtained using positivity in a crucial way.
\begin{lemma}\label{pp}
We consider a measurable subset $E$ of $\C$, a complex-valued
measure $m$ on $E$, and a bounded and continuous function
$M(.)\in C(E;\LH)$ defined on $E$ such that the operator $M(\sigma)$ is self-adjoint and positive for each $\sigma\in E$.
We also assume that there exist a positive bounded measure $dn$ with $|dm|\leq
dn$ on $E$.
Then the map
\begin{equation*}
r\mapsto g_1(r), \quad g_1(r)=\int_{E}r(\sigma) \,M(\sigma)\,dm(\sigma)
\end{equation*}
 is completely bounded from the
algebra ${\mathcal R}(E)$ into the algebra $\LH$, with complete
bound $\|g_1\|_{cb}\leq
\big\|\int_{E}M(\sigma)\,dn(\sigma)\big\|$.
\end{lemma}
\begin{lemma}\label{pm}
We consider a measurable subset $E$ of $\C$, a complex-valued
measure $m$ on $E$, and a  bounded and continuous function
$M(.)\in C(E;\LH)$ defined on $E$ with bounded invertible operator values.
We assume that there exist a continuous function $N(.)\in
C(E;\LH)$ defined on $E$ with bounded self-adjoint operator
values, a positive number $\alpha$ and a positive bounded measure $dn$
such that
 $\frac12 (M(\sigma)+M(\sigma)^*)\geq N(\sigma)\geq\alpha>0$, for all $\sigma\in E$ and $|dm|\leq dn$ on $E$. Then the map
 \begin{equation*}
r\mapsto g_2(r), \quad g_2(r)=\int_{E}r(\sigma) \,(M(\sigma))^{-1}dm(\sigma)
\end{equation*}
 is well-defined and completely bounded from the algebra ${\mathcal R}(E)$
into the algebra $\LH$, with complete bound $\|g_2\|_{cb}\leq
\big\|\int_{E}(N(\sigma))^{-1} dn(\sigma)\big\|$.
\end{lemma}
\begin{proof}
We only give the proof of Lemma \ref{pm}, the other being easier
and similar. We consider $R(\sigma)=\big(r_{ij}(\sigma)\big)\in
M_n(\C)$, where the components $r_{ij}(.)$ are rational functions
bounded on $E$ with values in $\C$. We associate to $R$, the
operator $G_2(R)$ on $H^n$ defined by $ G_2(R)_{ij}=g_2(r_{ij})$.
To say that $g_2(r)$ is completely bounded  with complete bound
$\|g_2\|_{cb}$ means that
\begin{equation*}
 \|G_2(R)\|_{H^n\to H^n}\leq \|g_2\|_{cb} \|R\|_{E},
\end{equation*}
for all such $R$ and all values of $n$.

Without loss of generality, we may assume that $N(\sigma)=\frac12
(M(\sigma)+M(\sigma)^*)$. Then we define $B(\sigma)
=N(\sigma)^{-1/2}$, so that we can write
$M(\sigma)=B(\sigma)^{-1}(I\!+\!iD(\sigma))B(\sigma)^{-1}$, with a
self-adjoint operator $D(\sigma)$. Let us consider $u$ and $v\in
H^n$, with $\|u\|=\|v\|=1$, then we have
\begin{align*}
 \langle G_2(R)u,v\rangle&=\int_{E}\langle (R(\sigma) \otimes M(\sigma)^{-1})u,v\rangle\,dm(\sigma)\\
 &=\int_{E}\langle R(\sigma) \otimes(I\!+\!iD(\sigma))^{-1}B(\sigma)u,B(\sigma)v\rangle\,dm(\sigma).
\end{align*}
We note that $\|(I\!+\!iD(\sigma))^{-1}\|\leq1$, whence
\begin{equation*}
 \|R(\sigma) \otimes(I\!+\!iD(\sigma))^{-1}\|_{H^n\to H^n}\leq  \|R\|_E.
\end{equation*}

This yields
\begin{align*}
 |\langle G_2(R)u,v\rangle|&\leq\|R\|_E \int_{E}\|B(\sigma)u\|\,\|B(\sigma)v\|\,dn(\sigma)\\
 &\leq\frac12\|R\|_E \Big(\int_{E}\|B(\sigma)u\|^2dn(\sigma)+\int_{E}\|B(\sigma)v\|^2dn(\sigma)\Big).
\end{align*}
But we have
\begin{align*}
\int_{E}\|B(\sigma)u\|^2dn(\sigma)&=\Big\langle\int_{E}B^2(\sigma)\,dn(\sigma)u,u\Big\rangle
\leq\Big\|\int_{E}(N(\sigma))^{-1}\,dn(\sigma)\Big\|,
\end{align*}
and the same inequality with $v$ in place of $u$. Finally we have
obtained
\begin{align*}
 |\langle G_2(R)u,v\rangle|&\leq\|R\|_E \Big\|\int_{E}(N(\sigma))^{-1}\,dn(\sigma)\Big\|,
 \end{align*}
for all $u$, $v$, with $\|u\|=\|v\|=1$, which completes the proof
of the lemma.
\end{proof}

As an application of Lemma~\ref{pp}, we will prove an estimate for the Poisson term defined in the introduction.

\begin{corollary} \label{decadix}
(a) \quad Let the closed disk $D$ of the Riemann sphere be a spectral set
for $A\in \LH$, and let $\Gamma\subset \partial D$. With the above notation, the map
\begin{equation*}
f\mapsto g(f):=\int_{\Gamma} f(\sigma) \mu(\sigma,A,D)\,ds
\end{equation*}
is completely bounded from the algebra ${\mathcal R}(D)$ into
$\LH$, with complete bound $\leq 1$.

(b) \quad  Let $A\in \mathcal L(H)$, and consider the
intersection $X=D_1\cap D_2\cap\dots\cap D_n$ of $n$ disks of the
Riemann sphere $\overline \C$, each of them being spectral for
$A$. With the above notation, the map
$$f\mapsto g_{p}(f) =\sum_{j=1}^n \int_{X\cap \partial D_{j}}f(\sigma)\,\mu(\sigma,A,D_{j})\,ds$$
is completely bounded from the algebra ${\mathcal R}(X)$ into
$\LH$, with complete bound $\leq n$.
\end{corollary}
\begin{proof}
  For (a) we want to apply Lemma \ref{pp} with
  $E=\Gamma$, $M(\sigma)=\mu(\sigma,A,D)$, $dm=ds$ and $dn=ds$.
  Hence it is sufficient to verify that
  \begin{align}
      & \label{decadix1}
      \mu(\sigma,A,D)\geq0, \qquad\hbox{ for all $\sigma\in \partial
      D$,}
     \\ &\label{decadix2}
       \int_{\partial D}\mu(\sigma,A,D)\,ds=1,  \qquad\hbox{
          (identity on $H$).}
  \end{align}
  For a proof of (\ref{decadix1}), consider for instance
  a disk of the form $D=\{|z\!-\!\omega|\geq r\}$ (proofs for compact
  disks and half-planes are similar, we omit here the details).
  Since $D$ is a spectral set for $A$, we have
  $(A\!-\!\omega)(A^*\!-\!\bar\omega)-r^2\geq0$.
  Then with the parametrization $\sigma=\omega+re^{-is/r}$ we get
\begin{align*}
\mu(\sigma,A,D)=-\tfrac{1}{2\pi }\big(
r(r\!-\!e^{is/r}(A\!-\!\omega))^{-1}+r(r\!-\!e^{-is/r}(A^*\!-\!\bar\omega))^{-1}-1\big)
\qquad\\ =\tfrac{1}{2\pi
}(r\!-\!e^{is/r}(A\!-\!\omega))^{-1}\big((A\!-\!\omega)(A^*\!-\!\bar\omega)-r^2)\big)(r\!-\!e^{-is/r}(A^*\!-\!\bar\omega))^{-1},
\end{align*}
which shows that $\mu(\sigma,A,D)\geq0$. Finally, the relation
(\ref{decadix2}) easily follows from the Cauchy formula. The proof of (b) follows by applying part (a) to each term of the sum.
\end{proof}

\section{The annulus case: proof of Theorem~\ref{thm1}}\label{s2.1}

We start by establishing the upper bounds in Theorem~\ref{thm1},
implying that Theorem~\ref{thm0} holds for $n=2$ in the special
case of an annulus.

\begin{proof}[Proof of the upper bounds in Theorem~\ref{thm1}]
  We have by assumption that the disks
  \begin{align*}
      &
       D_1 = \{ z \in \mathbb C : |z|\leq R \}, \quad
       \mbox{and} \quad
       D_2 = \{ z \in \mathbb C : |z|\geq R^{-1} \} \cup \{ \infty\}
  \end{align*}
  are spectral for $A$, and $X=X(R)=D_1\cap D_2$.

  Taking into account the complete bound $n=2$ of the map $f\mapsto g_p(f)$, it suffices
  to show that  the map $f\mapsto g_r(f)$ is completely bounded by $(R\!+\!1)/\sqrt{R^2\!+\!R\!+\!1}$.
  Recall that
    $$
  g_r(f)= \int_{ \partial D_1} f(\sigma) \,
    \nu(\sigma,A,D_{1}) d\sigma+  \int_{ \partial D_2} f(\sigma) \,
    \nu(\sigma,A,D_{2}) d\sigma,
    $$
    and that the orientations of $\partial D_{1}$ and $\partial D_{2}$ are opposite.
    We have on $\partial D_{1}$ the identity $\bar \sigma=R^{2}/\sigma$, thus
     \begin{align*}
    \nu(\sigma,A,D_{1})\, d\sigma&=\frac{1}{2\pi i} \,(\bar\sigma-A^*)^{-1}d\bar\sigma+\sigma^{-1}d\sigma)\\
    &=\frac{1}{2\pi i} A^*(\sigma A^*\!-\!R^2)^{-1}d\sigma.
     \end{align*}
     Note that $(\sigma A^*\!-\!R^2)^{-1}$ is holomorphic
     in $\sigma$, for $|\sigma|\leq R$. Therefore we can
     deform the boundary $\partial D_1$ into the unit circle
     and get
     $$
 \int_{\partial D_1} f(\sigma) \,
    \nu(\sigma,A,D_{1}) d\sigma=\int_{|\sigma|=1} f(\sigma) \,
    \nu(\sigma,A,D_{1}) d\sigma.
    $$
    Similarly, with $\nu(\sigma,A,D_2)=\frac{1}{2\pi i} A^*(\sigma A^*\!-\!R^{-2})^{-1}$, and keeping for the unit circle the orientation
    of $\partial D_{1}$,
       $$
 \int_{\partial D_2} f(\sigma) \,
    \nu(\sigma,A,D_{2}) d\sigma=-\int_{|\sigma|=1} f(\sigma) \,
    \nu(\sigma,A,D_{2}) d\sigma
    $$
    and
       $$
  g_r(f)=\int_{|\sigma|=1} f(\sigma) \,
   ( \nu(\sigma,A,D_{1}) -\nu(\sigma,A,D_{2}) )d\sigma.
    $$
Setting
  $\sigma(\theta)=e^{i\theta}$, we get
  \begin{align} \label{nu_integral_annulus}
      g_r(f) =- \frac{R^2-R^{-2}}{2\pi} & \int_{0}^{2\pi} f(e^{i\theta})
      M(\theta,A^*)^{-1} d\theta,
      \\& \nonumber
      \mbox{where} \quad  M(\theta,A^*) := R^2+R^{-2} - e^{i\theta} A^* - e^{-i\theta}
      (A^{-1})^*.
  \end{align}

We write $A^*=UG$, with a unitary operator $U$ and a positive operator $G$. Then the assumption ``$D_1$ and
$D_2$ spectral sets for $A$'' reads $R^{-1}I\leq G\leq RI$. Setting
$\rho=\frac{R+R^{-1}}{2}$, we have
\begin{equation*}
\|G+G^{-1}-(\rho+1)I\|\leq\max_{R^{-1}\leq x\leq R}
|x\!+\!x^{-1}\!-\!\rho\!-\!1| = \rho-1.
\end{equation*}
This yields, for the self-adjoint part of $M(\theta,A^*)$,
\begin{align*}
\Re M(\theta,A^*)&=R^2+R^{-2}-(\rho\!+\!1)
\Re(e^{i\theta}U)+\Re(e^{i\theta}U(G\!+\!G^{-1}\!\!-\!\rho\!-\!1))\\
&\geq R^2+R^{-2}-(\rho\!+\!1) \Re(e^{i\theta}U)-\rho+1\geq
R^2+R^{-2}-2\rho>0.
\end{align*}
We then apply Lemma \ref{pm} with
\begin{align*}
N(\theta)= R^2+R^{-2}-(\rho\!+\!1) \Re(e^{i\theta}U)-\rho+1,
\quad dm= dn=\frac{R^2-R^{-2}}{2\pi} d\theta,\quad
\end{align*}
and get that $g_r$ is completely bounded with a complete bound
\begin{align*}
\|g_r\|_{cb}&\leq
\frac{R^2-R^{-2}}{2\pi}\Big\|\int_0^{2\pi}\big(R^2+R^{-2}-(\rho\!+\!1)
\Re(e^{i\theta}U)-\rho+1 \big)^{-1}d\theta\Big\|.
\end{align*}
Elementary calculations show that, if we set,
\begin{align*}
h(z)=\int_0^{2\pi}\frac{d\theta}{R^2+R^{-2}-\rho+1-(\rho\!+\!1)
(e^{i\theta} z\!+\!e^{-i\theta}z^{-1})/2} ,
\end{align*}
then we have
\begin{align*}
h(e^{i\varphi})&=
\frac{2\pi}{\sqrt{R^2+R^{-2}-R-R^{-1}}\,\sqrt{R^2+R^{-2}+2}}\\
&=
\frac{2\pi}{R^2-R^{-2}}\frac{R-R^{-1}}{\sqrt{R^2+R^{-2}-R-R^{-1}}}\\
&=\frac{2\pi}{R^2-R^{-2}}\
\frac{\sqrt{R}+\sqrt{R^{-1}}}{\sqrt{R+R^{-1}+1}}=h(1).
\end{align*}
This yields that $h(U)=h(1)$ and thus
\begin{align*}
\|g_{r}\|_{cb}\leq \frac{R^2-R^{-2}}{2\pi}
\|h(U)\|=\frac{\sqrt{R}+\sqrt{R^{-1}}}{\sqrt{R+R^{-1}+1}} =
\frac{R+1}{\sqrt{R^2+R+1}},
\end{align*}
as required for the upper bound claimed in Theorem~\ref{thm1}.
\end{proof}

\begin{proof}[Proof of the lower bound in Theorem~\ref{thm1}]
For $t \in \R$, let
   $$A(t) = \left(\begin{array}{cc}
            1 & t  \\
            0 & 1
        \end{array}\right) \
\quad\hbox {with inverse }
A(t)^{-1} =  A(-t)$$
acting on the Hilbert space $\C^2$. For $t = R \!-\! R^{-1}$ we have
$\|A(t)\| = \|A(t)^{-1}\| = R$ (compare with \cite[p.
152]{paul}).  For a rational function $f $ we have
\begin{eqnarray*}%\label{eq2}
    f(A(t)) =
\left(\begin{array}{cc}
            f(1) & t f'(1)  \nonumber \\
            0 & f(1)
        \end{array}\right),
\end{eqnarray*}
which implies $\|f(A(t))\|\geq t|f'(1)|$. This shows that $$
   K(R)\geq (R \!-\! R^{-1})\,
    \sup\{|f'(1)|/\|f\|_{X(R)}\,; f\in
    \Rat(X)\}. $$
It follows from the computation of the infinitesimal
Carath\'eodory
    pseudodistance
on the annulus by Simha \cite[Example~(5.3)]{simha} that
\begin{equation*}
\sup\{|f'(1)|/\|f\|_{X(R)}\,; f\in \Rat(X), f(1)=0\}
= \frac{2}{R} \prod_{n=1}^\infty \Bigl(\frac{1-R^{-8n}}{1-R^{4-8n}}
    \Bigr)^2.
\end{equation*}
Thus
\begin{align*}
   K(R)\geq \gamma(R)&:=
   2(1 \!-\! R^{-2}) \prod_{n=1}^\infty
   \Bigl(\frac{1-R^{-8n}}{1-R^{4-8n}}
    \Bigr)^2)\\
    &=\lim_{k\to \infty }\gamma_k(R), 
    \end{align*}
    with 
    \begin{align*}
      & \gamma_{k}(R):=\frac{2}{1+R^{-2}}\prod_{n=1}^k \Bigl(\frac{1-R^{-8n}}{1-R^{4-8n}}\,\frac{1-R^{-8n}}{1-R^{-4-8n}}
         \Bigr).
    \end{align*}
    For $k=1$ we have
    \begin{align*}
    \gamma_1(R)=\psi(R^2),\quad\hbox{with }\ \psi(t):=\frac{2t(1\!+\!t^2)^2}{(1\!+\!t)(1\!+\!t^2\!+\!t^4)}.
    \end{align*}
It is easily seen that $\psi'(t)>0$, which shows that $\gamma_1(\cdot)$ is an increasing function. We have  $\gamma_1(1)=4/3$, and $\gamma_1(\infty)=2$.
We deduce that $K(R)\geq \gamma_1(R)> \gamma_1(1)=4/3$.
Note also that, for each $R$ the sequence $\gamma_k(R)$ is increasing with $k$, with upper bound 2.
\end{proof}
   
\begin{rem}\label{rem_bound}
    We have
    \begin{align*}
     \lim_{R \to 1} \gamma(R)=\prod_{n=1}^\infty\frac{4n^2}{4n^2-1}= \frac{\pi}{2} , \qquad
              \lim_{R \to \infty} \gamma(R)= 2 .
    \end{align*}
Numerical computations show that $\gamma(R)$ increases with $R$, which leads to a better lower bound  $\gamma(R)\geq \gamma(1)=\pi/2$.    For the sake of comparison, we have drawn in
    Figure~\ref{figure1} the different lower and upper bounds for
    $K(R)$ as a function of $R$. Returning to the second question of
    Shields mentioned in the introduction, we see that an
    optimal constant independent of $R$ for our annulus $X(R)$
    has to lie between $2$ and $2\!+\!2/\sqrt3$.
\end{rem}

\begin{figure}[t]
 \centerline{\epsfig{file=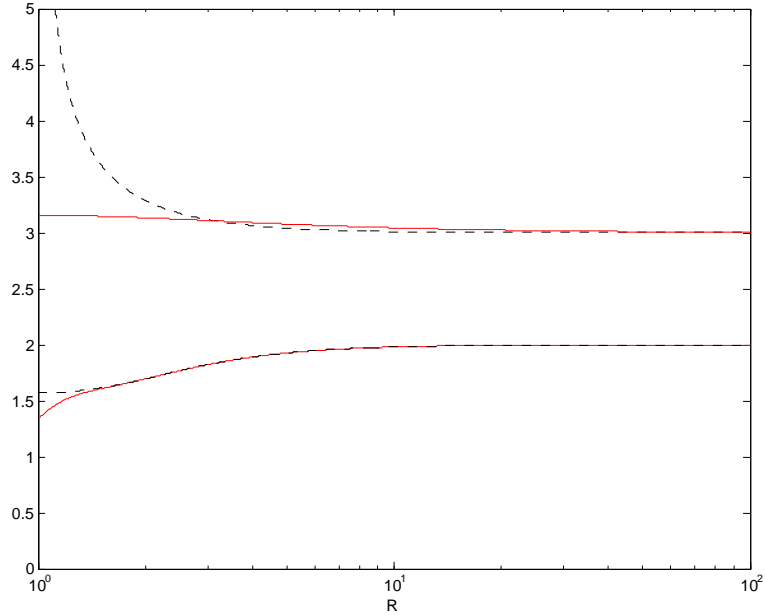,scale=.6}}
 \caption{{\it
 The different upper and lower bounds for $K(R)$ for an
 annulus $X(R)$ as a function of $R\in (1,+\infty)$:
 from above to below, we find the upper
 bound of Shields (dashed line, which is unbounded), the upper bound
 of Theorem~\ref{thm1} (solid line), and the lower bounds $\gamma_1(R)
 <\gamma(R)$ from the proof of Theorem~\ref{thm1}.
 % we find in a neighborhood of $R=1$ the upper
 %bound of Shields (which is unbounded), the upper bound
 %of Theorem~\ref{thm1}, and the lower bounds $\gamma_1(R)<\gamma(R)$ from the proof
 %of Theorem~\ref{thm1}.
 }}\label{figure1}
\end{figure}

\begin{rem}\label{rem_bound2}
We owe to a suggestion of Vern Paulsen an improvement of upper bounds of $K(R)$ for large values of $R$. More precisely, adapting the arguments of \cite[Theorem\,4.1]{PauSin}, we obtain 
\begin{equation*}
K(R)\leq\max(3,2\!+\!\psi(R))\quad\hbox{with} \quad\psi(R)=\sum_{n\geq1}\frac{4}{1\!+\!R^{2n}}.
\end{equation*}
This shows that $K(R)\leq3$, for $R\geq 2.0952978$. This estimate is better than the upper bound
$2+(R\!+\!1)/\sqrt{R^2\!+\!R\!+\!1}$ of Theorem \ref{thm1},
if $R\geq1.9878813$,  and better than the upper bound of Shields if $R\geq1.85443$.
\begin{proof}
Let $f$ be a rational function bounded by 1 in the annulus $X(R)$. We write
\begin{equation*}
f(z)=f_1(z)\!+\!f_2(z)\quad\hbox{with} \quad f_1(z)=\sum_{n\geq0}a_nz^n,\  f_2(z)=\sum_{n<0}a_nz^n.
\end{equation*}
It suffices to show that $\|f_1\|_{D_1}+\|f_2\|_{D_2}\leq \max(3,2\!+\!\psi(R))$. For that we may assume, without loss of generality that $a_0\geq0$. For each integer $n>0$ , each real $r\in[1/R,R]$, and $\varphi\in \R$, we have
\begin{align*}
a_nr^ne^{i\varphi}&=-\frac{1}{2\pi}\int_0^{2\pi}(1-e^{i\varphi}f(re^{i\theta}))e^{-in\theta}d\theta,\\
\overline{a_{-n}}r^{-n}e^{-i\varphi}&=-\frac{1}{2\pi}\int_0^{2\pi}(1-e^{-i\varphi}\overline{f(re^{i\theta})})e^{-in\theta}d\theta,
\end{align*}
and
\begin{align*}
a_nr^ne^{i\varphi}+
\overline{a_{-n}}r^{-n}e^{-i\varphi}=-\frac{1}{\pi}\int_0^{2\pi}\Re(1-e^{i\varphi}f(re^{i\theta}))e^{-in\theta}d\theta.
\end{align*}
We note that $\Re(1-e^{i\varphi}f(re^{i\theta}))\geq0$, thus
\begin{align*}
|a_nr^ne^{i\varphi}+
\overline{a_{-n}}r^{-n}e^{-i\varphi}|\leq\frac{1}{\pi}\int_0^{2\pi}\Re(1-e^{i\varphi}f(re^{i\theta}))d\theta=2(1\!-\!\Re(e^{i\varphi}a_0)).
\end{align*}
With a judicious choice of $\varphi$ we obtain
\begin{align*}
|a_n|r^n+
|a_{-n}|r^{-n}\leq2(1\!-\!a_0)), \quad \hbox{for all  } r\in [1/R,R].
\end{align*}
Writing this inequality with the values $r=R$ and  $r=1/R$, and then adding, we deduce
\begin{align*}
(|a_n|\!+\!|a_{-n}|)(R^n\!+\!R^{-n})\leq4(1\!-\!a_0)).
\end{align*}
We note that, from the maximum principle,
\begin{align*}
\|f_1\|_{D_1}&=\max_{|z|=R}(|f(z)\!-\!f_2(z)|)\leq 1+\max_{|z|=R}(|f_2(z)|)\\
&\leq 1+\sum_{n\geq1}|a_{-n}|R^{-n}.
\end{align*}
Similarly we have
\begin{align*}
\|f_2\|_{D_2}&=\max_{|z|=1/R}(|f(z)\!-\!f_1(z)|)\leq 1+a_0+\sum_{n\geq1}|a_n|R^{-n},
\end{align*}
and finally
\begin{align*}
\|f_1\|_{D_1}+\|f_2\|_{D_2}&\leq 2+a_0+\sum_{n\geq1}(|a_n|\!+\!|a_{-n}|)R^{-n}\\
&\leq 2+a_0+4(1\!-\! a_0)\sum_{n\geq1}(1\!+\!R^{2n})^{-1}.
\end{align*}
The result follows by noticing that $a_0\in [0,1]$.
\end{proof}

\end{rem}

\section{The general case of two closed disks of the Riemann sphere}\label{s2.5}
Recall our basic assumption that for each spectral set $D$ for a given operator $A$, the spectrum $\sigma(A)$ is included in the interior of $D$.
Recall also that we have defined the residual kernel in a point $\sigma$ on the boundary of a disk $D$ centered in $\omega$ by
\begin{eqnarray*}
\nu(\sigma,A,D)\,d\sigma=\frac{1}{2\pi i}(\sigma\!-\!A)^{-1}d\sigma-\mu(\sigma,A,D)ds =
\frac{1}{2\pi i}(\bar\sigma\!-\!A^*)^{-1}d\bar\sigma+\frac{d\sigma}{\sigma-\omega}.
 \end{eqnarray*}
 If $\partial D=\{z\,;|z\!-\!\omega|=r\}$, then $\bar\sigma\!-\!\bar\omega=r^2/(\sigma\!-\!\omega)$, thus
 \begin{eqnarray*}
\nu(\sigma,A,D)&=&\frac{1}{2\pi i}\Big(\frac{1}{\sigma-\omega}-
\frac{r^2}{\sigma-\omega}(r^2-(\sigma\!-\!\omega)(A^*-\bar\omega))^{-1}\Big)\\
&=&\frac{1}{2\pi i}(A^*-\bar\omega)\big((\sigma\!-\!\omega)(A^*-\bar\omega)\!-\!r^2\big)^{-1}.
 \end{eqnarray*}
 Note that $\big((z\!-\!\omega)(A^*-\bar\omega)\!-\!r^2\big)^{-1}$ is also well-defined, and analytic in $z$, if $z\in D$, under the assumption $\sigma(A)\subset$ int $\!(D)$. This allows us to extend the previous definition of $\nu$ by
 \begin{eqnarray*}
\nu(z,A,D):=\frac{1}{2\pi i}(A^*-\bar\omega)\big((z\!-\!\omega)(A^*-\bar\omega)\!-\!r^2\big)^{-1},
\quad\hbox{for}\ z\in D.
 \end{eqnarray*}
 Similarly, if $D$ is a half-plane $D=\Pi=\{z\,; \Re\big(e^{-i\theta}(z\!-\!a)\bigr)\geq0\}$, we use from now on the definition
 \begin{eqnarray*}
\nu(z,A,\Pi):=\big(z\!-\!a\!+e^{2i\theta}\!(A^*-\bar a)\big)^{-1},
\quad\hbox{for}\ z\in \Pi.
 \end{eqnarray*}

\medskip

 An important point in our study is the homographic invariance of the differential form
 \begin{equation*}
\big(\nu(z,A,D_j)-\nu(z,A,D_k)\big)\,dz.
\end{equation*}
More precisely we have the following result
\begin{lemma}\label{homo}
Let $D_j$ and $D_k$ be two disks of $\overline\C$, and let $A$ be an operator with spectrum $\sigma(A)$ contained in the interior of $D_j\cap D_k$. Then, for all M\"obius map $\varphi$ with poles off $\sigma(A)$, we have
 \begin{equation*}
\big(\nu(z,A,D_j)-\nu(z,A,D_k)\big)\,dz=\big(\nu(\zeta,B,\Delta_j)-\nu(\zeta,B,\Delta_k)\big)\,d\zeta,
\end{equation*}
where $\zeta=\varphi(z)$, $B=\varphi(A)$, $\Delta_j=\varphi(D_j)$, and $d\zeta=\varphi'(z)\,dz$.
\end{lemma}
\begin{proof} For the sake of the clarity we
will assume in the proof that $D_j$, $D_k$,
   $\Delta_j$ and $\Delta_k$ are not half-planes. The case of half-planes could be obtained either by a limiting argument, by replacing half-planes by suitable supersets being unbounded disks, or otherwise directly by slightly changing the arguments below. We omit the (elementary but tedious) details.
   
Recall that a M\"obius map of the Riemann sphere is an
automorphism $\varphi$ of $\overline \C$ which has the form
\begin{align*}
\varphi(z)=\frac{az+b}{cz+d}, \quad \hbox{with}\quad ad\neq bc,\  a,b,c,d\in \C.
\end{align*}
These transformations form a group generated by the translations:
$\varphi(z)=z+b$, the dilations centered in 0: $\varphi(z)=\lambda z$, $\lambda\neq 0$, and the
inversion-symmetry: $\varphi(z)=1/z$.
The lemma is easily verified if $\varphi$ is a translation or a dilation, therefore it suffices  to consider the case $\varphi(z)=1/z$.\\
 So, if $\omega_j$ and $r_j$ denote respectively the center and
   the radius of $D_j$, after the transformation $\varphi(z)=1/z$,
   we get a disk
   $\Delta_j=\varphi(D_{j})$ with center and radius
   \begin{equation*}
      o_j=
      \frac{\bar \omega_j}{|\omega_j|^2-r_j^2},
      \quad
       \rho_j=\frac{r_j}{|\omega_j|^2-r_j^2}.
   \end{equation*}
   Setting $\zeta=1/z$, $B=A^{-1}$,
   we have
   \begin{equation*}
      B^*\!-\!o_j=
      \frac{1}{|\omega_j|^2-r_j^2}\,
        ({A^{-1})}^*(|\omega_j|^2\!-\!r_j^2\!-\!\omega_jA^*),
    \end{equation*}
    and thus
   \begin{align*}
   (\zeta\!-\!o_j)(B^*\!-\!o_j)&=
   \frac{(A^{-1})^*}{z(|\omega_j|^2-r_j^2)^2}\,
    \big((|\omega_j|^2\!-\!r_j^2\!-\!\omega_jA^*)(|\omega_j|^2\!-\!r_j^2\!-\!\bar\omega_jz)
    -zr_j^2A^*\big)\\
    &=\frac {(A^{-1})^*}{z(|\omega_j|^2-r_j^2)}\,
    \big(|\omega_j|^2\!-\!r_j^2\!-\!\omega_jA^*+z(A^*\!-\!\bar\omega_j)\big),
 \end{align*}
 implying that
 \begin{align*}
    2\pi i\,\nu(\zeta,B,\Delta_j)
    &=z(|\omega_j|^2\!-\!r_j^2\!-\!\omega_jA^*)
    \big(|\omega_j|^2\!-\!r_j^2\!-\!\omega_jA^*+z(A^*\!-\!\bar\omega_j)\big)^{-1}
    \\
     &=z-z^2(A^*\!-\!\bar\omega_j)
    \big(|\omega_j|^2\!-\!r_j^2\!-\!\omega_jA^*+z(A^*\!-\!\bar\omega_j)\big)^{-1}\\
    &=z-2\pi i\, z^2  \nu(z,A,D_j).
 \end{align*}
 Using that $d\zeta=-z^{-2}dz$, we get
 \begin{equation*}
 \big(\nu(\zeta,B,\Delta_j) \!-\!
\nu(\zeta,B,\Delta_k)
 \big)\,d\zeta =
  \big(\nu(z,A,D_j)\!-\!\nu(z,A,D_k)\big)\,dz,
  \end{equation*}
  which completes the proof.
\end{proof}

We are now prepared to present our proof of Theorem~\ref{thm0} for
the case $n=2$ of the intersection $X=D_1\cap D_2$ of two
arbitrary closed disks of the Riemann sphere.

We warn again the reader that each disk is a
spectral set for the bounded operator $A$, and that the
spectrum of $A$ is interior to $X$. We have seen in
Corollary \ref{decadix} that the map $f\mapsto g_p(f)$ is
completely bounded, with constant $2$, it remains to
show that  the map $f\mapsto g_r(f)$ is completely bounded,
with constant $2/\sqrt3$. We have
\begin{equation}\label{residu}
g_r(f)=\int_{X\cap\partial D_1}f(\sigma)\,\nu(\sigma,A,D_1)\,d\sigma+
\int_{X\cap\partial D_2}f(\sigma)\,\nu(\sigma,A,D_2)\,d\sigma.
\end{equation}
Three cases may occur
\medskip

\noindent{ \bf Case 1.} $\partial D_1\cap \partial D_2=\emptyset$. Then there exist a M\"obius map $\varphi$ and a real $R>1$ such that
 \begin{align*}
  \Delta_1=\varphi(D_1)=\{z\in \C\,; |z|\leq R\}\quad\hbox{and}\quad
    \Delta_2=\varphi(D_2)=\{z\in \overline\C\,; |z|^{-1}\leq R\}.
 \end{align*}
 Note that the pole of the map $\varphi$, which is $\varphi^{-1}(\infty)$, is not in $D_1$ and thus does not belong to $\sigma(A)$.
 We introduce a median circle, $C_{12}:=\{z\in \overline\C\,; |\varphi(z)|=1\}$.

\medskip

  \noindent{ \bf Case 2.} $\partial D_1\cap \partial D_2=\{\alpha,\beta\}$, $\alpha\neq\beta$. Then there exist a M\"obius map $\varphi$ and a real $\theta\in(0,\pi/2)$ such that
   \begin{align*}
  \Delta_1&=\varphi(D_1)=\{z\in \C\,;\Re(e^{i\theta} z)\geq0\}\cup\{\infty\}\quad\hbox{and}\\
    \Delta_2&=\varphi(D_2)=\{z\in \C\,;\Re(e^{-i\theta} z)\geq0\}\cup\{\infty\}.
 \end{align*}
 The pole of the map $\varphi$ now is $\alpha$ or $\beta$ and does not belong to $\sigma(A)$.
 We introduce a median line, $C_{12}:=\{z\in \overline\C\,;\varphi(z)\in [0,\infty]\}$.

\medskip
  \noindent{ \bf Case 3.} $\partial D_1\cap \partial D_2=\{\alpha\}$. Then there exists a M\"obius map $\varphi$ such that
   \begin{align*}
  \Delta_1&=\varphi(D_1)=\{z\in \C\,;\Im z \leq 1\}\cup\{\infty\}\quad\hbox{and}\\
    \Delta_2&=\varphi(D_2)=\{z\in \C\,;\Im z\geq -1\}\cup\{\infty\}.
 \end{align*}
 The pole of $\varphi$ now is $\alpha$ which does not belong to $\sigma(A)$. We set
 $C_{12}:=\{z\in \overline\C\,;\varphi(z)\in \overline\R\}$.\medskip

In each case we can continuously deform, staying in the set $X$,
$X\cap\partial D_1$ onto $C_{12}$ as well as  $X\cap\partial D_2$
onto $C_{12}$.  We can choose $C_{12}$ with the same orientation
that $\partial D_1$, which is opposite to the orientation of
$\partial D_2$. Since $f(.)\nu(.,A,D_j)$ is holomorphic in $X$, we
deduce from \eqref{residu}
\begin{equation*}
g_r(f)=\int_{C_{12}}f(z)\,\big(\nu(z,A,D_1)-\nu(z,A,D_2)\big)\,dz.
\end{equation*}
Theorem \ref{thm0}, in the case $n=2$, then follows from the following lemma
\begin{lemma}\label{lem4}
Consider two distinct disks $D_1$ and $D_2$ of $\overline\C$, and
a bounded operator $A\in \LH$ such that $D_1$ and $D_2$ are
spectral for $A$. Then, for any subarc of the median line $\Gamma\subset C_{12}$, the map
\begin{align*}
      & f \mapsto g(f) = \int_\Gamma f(z)\,\big(\nu(z,A,D_1)-\nu(z,A,D_2)\big)\,dz,
\end{align*}
is completely bounded from the algebra $\Rat (D_1 \cap D_2)$ into
$\LH$, with complete bound $\leq 2/\sqrt 3$.
\end{lemma}

\begin{proof}
We use the M\"obius map
$\varphi$ previously described in the three possible cases and set $\zeta=\varphi(z)$, $h(\zeta)=f(z)$, $B=f(A)$.
Note that, since $\varphi$ is a bijection, the disks $\varphi(D_1)$ and $\varphi(D_2)$  are spectral for $B$, and the spectrum of $B$ is included in the interior of these disks. We get from Lemma \ref{homo}
\begin{align*}
   g(f) &= \int_\Gamma f(z)\,\big(\nu(z,A,D_1)-\nu(z,A,D_2)\big)\,dz\\
   &=\int_{\varphi(\Gamma)} h(\zeta)\,\big(\nu(\zeta,B,\Delta_1)-\nu(\zeta,B,\Delta_2)\big)\,d\zeta.
   \end{align*}
Therefore,  it suffices to show the lemma in the three following cases\,:
\medskip

\noindent $\bullet \, D_1 \cap D_2$ is an annulus $X(R)=\{z\in \C\,; R^{-1}\leq |z|\leq R\}$,

\noindent $\bullet \,
D_1 \cap D_2$ is a sector $X=\{z\in \C\,; \Re (e^{i\theta}z)\geq0$ and $\Re (e^{-i\theta}z)\geq0\}\cup\{\infty\}$,

\noindent $\bullet \,
D_1 \cap D_2$ is a strip $X=\{z\in \C\,; -1\leq\Im z\leq1\}\cup\{\infty\}$.

\medskip
The {\it annulus case} has been treated in Section \ref{s2.1},
except that the set $\{z\,; |z|=1\}$ is now replaced by a subset $\Gamma$.
The necessary modifications in the proof are rather minor, and so
will be omitted here.

\medskip

{\it The sector case.} We have $\Gamma\subset \R_+$ and
\begin{align*}
   g(f) &= \int_\Gamma f(x)\,\big(\nu(x,A,D_1)-\nu(x,A,D_2)\big)\,dx\\
   &=\frac{\sin 2\theta}{\pi}\int_0^\infty f(x) M(x)^{-1}{\mathbf 1}_{_{\Gamma}}(x)\,dx,
\end{align*}
with   $M(x)=A^*+2x\cos2\theta+x^2(A^*)^{-1}$.
The conditions $D_1$ and $D_2$ spectral sets for $A$ means that
$\langle Av,v\rangle\in D_1\cap D_2 \subset \{ z \in \C;
\Re(z)\geq 0 \}$ for all $v\in H$. Hence the operator
$B=\frac12(A\!+\!A^*)$ is a positive operator, and we have
$A=B^{1/2}(I\!+\!iC)B^{1/2}$, with a self-adjoint operator $C$
satisfying $\|C\|\leq1/\tan\theta$. Note that
 \begin{equation*}
\Re (I\!-\!iC)^{-1}\geq
\inf_{\lambda\in\sigma(C)}\Re\frac{1}{1-i\lambda}=
\inf_{\lambda\in\sigma(C)}\frac{1}{1+\lambda^2}\geq\sin^2\!\theta.
\end{equation*}
Hence we have $\Re (A^*)^{-1}\geq\sin^2\!\theta \,B^{-1}$ and
$\Re M(x)\geq x^2\sin^2\!\theta\, B^{-1}\!+2x\cos2\theta
+B$. We then apply Lemma \ref{pm} with
\begin{align*}
  & N(x)= x^2\sin^2\!\theta\, B^{-1}\!+2x\cos2\theta +B,
  \quad  dm={\mathbf 1}_{_{\Gamma}}dn,\quad dn=
\frac{\sin2\theta}{\pi} dx.
\end{align*}
Setting \begin{align*}
      & h(z,\theta) =
      \frac{\sin2\theta}{\pi}\int_0^{\infty}\frac{dx}{x^2\sin^2\!\theta/z+2x\cos2\theta+z}
\end{align*}
for $z \geq 0$, we observe that $h(z,\theta)=h(1,\theta)$ does not
depend on $z$, and Lemma \ref{pm} allows us to conclude that $g$
is completely bounded with a complete bound
\begin{align*}
\|g\|_{cb}&\leq \| h(B,\theta) \| = h(1,\theta)
\leq\max_{\theta'\in [0,\pi/2]} h(1,\theta')=h(1,\pi/2)=2/\sqrt{3}.
\end{align*}
We should notice that the estimate $h(1,\theta) \leq h(1,\pi/2)$,
though sufficient for the purpose of Theorem~\ref{thm0}, is quite
rough. Alternately, one may evaluate the integral expression for
$h(1,\theta)$. Also, our Corollary~\ref{decadix}\,(b) may be
improved for a sector. We refer the reader to \cite{bcd} and
\cite[Chapter~4]{becr2} for a discussion of optimized upper bounds
as a function of $\theta$.

{\it The strip case.} We now have $\Gamma\subset \R$ and
\begin{align*}
   g(f) &= \int_\Gamma f(x)\,\big(\nu(x,A,D_1)-\nu(x,A,D_2)\big)\,dx\\
   &=\frac{2}{\pi}\int_{-\infty}^{+\infty} f(x) M(x)^{-1}{\mathbf 1}_{_{\Gamma}}(x)\,dx,
\end{align*}
with   $M(x)=(x\!-\!A^*)^2+4$. Since $D_1$ and $D_2$ are
spectral sets for the operator $A$, we have
$|\Im\langle Av,v\rangle|\leq 1$ for all $v\in H$ with $\|v\|=1$. Hence, if we denote
$B=\frac12(A\!+\!A^*)$ and write $A=B\!+\!iC$, the self-adjoint operator $C$
satisfies $\|C\|\leq1$ and
 \begin{equation*}
\Re M(x)=(x\!-\!B)^2+4-C^2\geq N(x):=(x\!-\!B)^2+3.
\end{equation*}
We then deduce from Lemma \ref{pm} that $g$ is completely bounded with a
complete bound
\begin{align*}
\|g\|_{cb}&\leq \| h(B) \|=2/\sqrt{3},
\end{align*}
where
\begin{align*}
h(\lambda)=\frac{2}{\pi}\int_\R\frac{dx}{(x\!-\!\lambda)^2+3}=\frac{2}{\sqrt3}.
\end{align*}
\end{proof}

\begin{rem}\label{remm}
The key points in the proof of this lemma has been the homographic
invariance of the differential forms, and the choice of the median
lines as paths of integration.
 For our proof for more than two disks presented below, it
 will be useful to give a
 more abstract definition of median lines in terms of sets of
 equal "distance" to some parts of the boundary of $X$:
given a disk $D$ of the Riemann sphere, we introduce the following
notation
%variant of the infinitesimal Carath\'eodory metrics
\begin{equation}\label{caratheodory2}
d(z,D)=\left\{\begin{array}{l} (|r^2-|z\!-\!\omega|^2|)/r,\, %quad
\hbox{ if }D\hbox{ is a disk with center }\omega\hbox{ and radius
} r,\\ \Re (e^{-i\theta}(z\!-\!a)),\, %quad
\hbox{ if }D\hbox{ is
the half-plane } \{z\,; \Re (e^{-i\theta}(z\!-\!a))\geq0\}.
\end{array}     \right.
\end{equation}
By means of elementary computations we have
 \begin{equation*}
     \frac{1}{d(z,D)}=
     \frac{|\phi'(z)|}{1-|\phi(z)|^2} , \quad \hbox{for}\ z\in D,
\end{equation*}
with the Riemann conformal map $\phi$ from $D$
 onto the unit disk $\mathbb D$. As a consequence, $1/d(z,D)$ is
 the infinitesimal Carath\'eodory pseudodistance.
 In particular, it becomes obvious that
 \begin{align}\label{caratheodory3}
      &
   d(\varphi(z),\varphi(D)) =|\varphi'(z)|\, d(z,D),
\end{align}
for any M\"obius map $\varphi$.

Note also that $d(z,D)\to 0$, if $z$ tends to the boundary of
$D$.\\ Now, if $D_j$ and $D_k$ are two disks with non empty
intersection, we introduce their median line $C_{jk}$
 \begin{equation*}
 C_{jk}:=\{z\in D_j\cap D_k\,; d(z,D_j)=d(z,D_k)\}.
\end{equation*}
From (\ref{caratheodory3}) it follows that, if $C_{jk}$ is the
median line of $D_j$ and $D_k$, and if $\varphi$ is a M\"obius
map, then $\varphi(C_{jk})$ is the median line of $\varphi(D_j)$
and $\varphi(D_k)$. Also, the circle $\{z\,; |z|=1\}$ is the
median line in the annulus case $X(R)$, the positive real semi
axis is the median line in the sector case, the real axis is the
median line in the strip case. This shows that the present
definition effectively coincides with the notion of median line
used in the beginning of this section with the curve
$C_{12}$.
\end{rem}

\section{Completing the proof of Theorem~\ref{thm0} }\label{s3}
We can now present the proof of the general case of Theorem~\ref{thm0}.
\begin{proof}[Proof of Theorem~\ref{thm0} for the case $n \geq 3$] We consider $n$ closed disks $D_1$, $D_2$,\dots,$D_n$ of the
Riemann sphere, with  $D_j\nsubseteq D_k$, for all $j\neq k$, such
that $X=D_1\cap D_2\cap...\cap D_n$ has a non-empty interior
(these two assumptions can be added without loss of generality in
the statement of Theorem~\ref{thm0}). Given $n$ points
$z_1,...,z_n$ in the complex plane, the classical Voronoi
tesselation of $X$ consists in dividing up the set $X$ in sets
$X_j$ of points $\sigma\in X$ such that $\sigma$ is at least as
close to $z_j$ as to any of the other points $z_k$ for $k\neq j$.
In order to find the new integration paths in our representation
formula for the residual term, we will construct a similar Voronoi
tesselation, but use the distance to the boundaries $\partial D_j$
instead of the distance to points. In our case, this distance will
be measured in terms of $d$, the reciprocal of the infinitesimal
Carath\'eodory pseudodistance (see Remark \ref{remm}). We recall
that $d(\sigma,D_j)\to 0$ if $\sigma\in X$ tends to $\partial
D_j$.

To be more precise, let \begin{align*}
      & X_j=\Bigl\{\sigma\in X : d(\sigma,D_j)=\min_{k=1,...,n}
      d(\sigma,D_k) \Bigr\}.
\end{align*}
By construction, the union $X_1\cup ... \cup X_n$ equals our set
$X$. We notice a difference between the present case and the $n=2$
case: the intersection $X_j \cap X_k$ for $j\neq k$ is a subset of
the median line $C_{jk}$ of the disks $D_j$ and $D_k$, which, in
general is a proper subset, since the median line $C_{jk}$ does
not need to be a subset of $X_j$, see
Figures~\ref{figure2}--\ref{figure4}. This situation does not
occur in the case $n=2$ of two disks.

\begin{figure}[t]
 \centerline{\epsfig{file=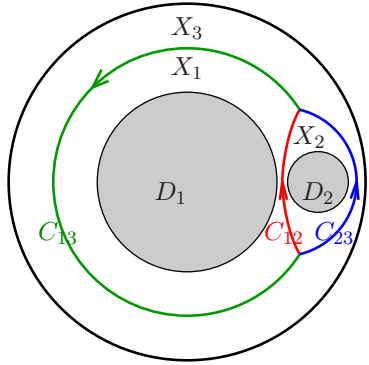,scale=.9}}
\caption{The tesselation for $D_1,D_3$
forming an annulus and $D_2$ exterior of a disk with isolated
boundary.}\label{figure2}
\end{figure}
\begin{figure}[t]
\centerline{\epsfig{file=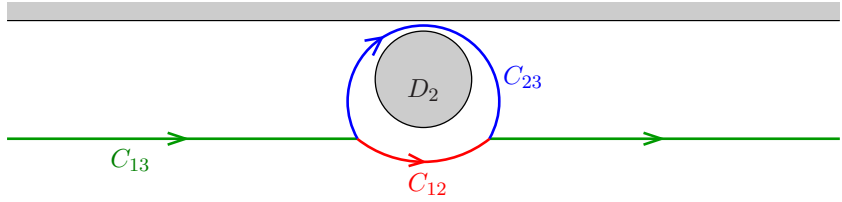,scale=.9}}
\caption{The
tesselation for $D_1,D_3$ half-planes with parallel boundaries and
$D_2$ exterior of a disk with with isolated
boundary.}\label{figure3}
\end{figure}
\begin{figure}[t]
 \centerline{\epsfig{file=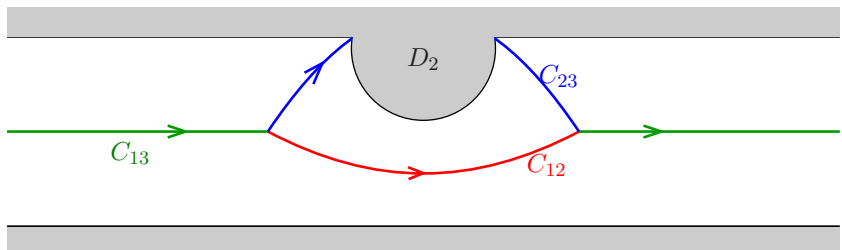,scale=.9}}
\caption{The
tesselation for $D_1,D_3$ half-planes with parallel boundaries and
$D_2$ exterior of a disk with boundary intersecting that of
$D_3$.}\label{figure4}
\end{figure}

In order to make the situation clearer in the general case $n\geq
3$, recall that the boundary of $X$ may be written as a union of
the sets $\partial D_j \cap X$ for $j=1,...,n$, where any two of
these sets have at most $2$ points in common (otherwise two of the
disks $D_j$ would have identical boundary, in contradiction to the
assumptions on the disks mentioned above). From the asymptotic
behavior of $d(\sigma,D_j)$ for $\sigma
\in X$ approaching $\partial D_j$ it follows that $\partial D_j
\cap X=\partial X \cap X_j$ for $j=1,...,n$, that is, $X_j$ can be
considered as a kind of closed neighborhood in $X$ of $\partial
D_j \cap X$.

It is now important to observe that, by construction, the boundary
of $X_j$ may be written as the union of $\partial D_j \cap X$ and
of $X_j \cap X_k$ for $k\neq j$ or, in other words, of the parts
of the median lines $C_{jk}$ for $k \neq j$ being a subset of
$X_j$. Again, any two of these pieces of $\partial X_j$ being all
circular arcs have at most two points in common, since a
non-trivial part of $\partial D_j$ cannot be a (part of a) median
line for disks $D_j$ and $D_k$ for $k\neq j$, and two medians
$C_{jk}$ and $C_{j\ell}$ only have more than two points in
common for $k,\ell \neq j$ if $D_k=D_\ell$.

Therefore, we have shown that, for a function analytic in $X_j
\subset X$, the path of integration $X \cap \partial D_j$ can be
replaced by the union of paths formed by the parts of the median
lines $C_{jk}$ for $k=1,...,n$, $k\neq j$, which have a non-empty
intersection with $X_j$ (and thus also with $X_k$). Since the
orientation of these median lines is changed if we permute $j$ and
$k$, we may conclude that, for any $f\in \Rat(X)$,
\begin{align*}
      g_r(f) &= f(A) -\sum_{j=1}^n \int_{X \cap \partial D_j} f(\sigma)
      \mu(\sigma,A,D_j) ds
      \\&= \sum_{j,k=1,j \neq k}^n
      \int_{C_{jk} \cap X_j \cap X_k} f(\sigma)
      \Bigl( \nu(\sigma,A,D_j) - \nu(\sigma,A,D_k) \Bigr) d\sigma.
\end{align*}
Thus Theorem~\ref{thm0} follows from Corollary~\ref{decadix} and
Lemma~\ref{lem4}.
\end{proof}

\bibliographystyle{amsplain}

\begin{thebibliography}{99}

\bibitem{agler}(794373)
  \newblock J.\,Agler, 
  \newblock \emph{Rational dilation on an annulus},
  \newblock Ann. of Math. (2) {\bfseries 121} (1985), 537--563.

 \bibitem{AHR} (2375060)
   \newblock J.\,Agler, J.\,Harland and B.J.\,Raphael, 
     \newblock  \emph{Classical function theory, operator dilation theory, and machine computation on multiply-connected domains}, 
     \newblock Mem. Amer. Math. Soc.  {\bfseries 191}  (2008),  no. 892, viii+159 pp.

\bibitem{bbc} 
    \newblock C.\,Badea,  B.\,Beckermann,  and M.\,Crouzeix,
 \newblock \emph{$K$-spectral sets and intersections of disks of the Riemann
sphere,} 
    \newblock manuscript, \arXiv:0712.0522v1 .

\bibitem{bcd} (2207801)
    \newblock C.\,Badea, M.\,Crouzeix and B.\,Delyon,
 \newblock \emph{Convex domains and $K$-spectral sets, }
  \newblock  Math. Z. {\bfseries 252} (2006), 345--365.

\bibitem{becr} (2223270)
  \newblock B.\,Beckermann and M.\,Crouzeix, 
   \newblock \emph{A lenticular version of a von Neumann inequality, }
    \newblock  Arch. Math. (Basel), {\bfseries 86} (2006), 352--355.

\bibitem{becr2} (2325887)
  \newblock B.\,Beckermann and M.\,Crouzeix, 
  \newblock \emph{Operators with numerical range in a conic domain,}
  \newblock  Arch. Math. (Basel), {\bfseries 88} (2007), 547-559.

\bibitem{crde} (2029717)
  \newblock M.\,Crouzeix, B.\,Delyon, 
   \newblock \emph{Some estimates for analytic functions of strip or sectorial operators,}
   \newblock  Arch. Math. (Basel), {\bfseries 81} (2003), 553-566.

 \bibitem{crzx} (2297040)
    \newblock M.\,Crouzeix, 
  \newblock \emph{Numerical range and functional calculus in Hilbert space, }
    \newblock J. Funct. Anal., {\bfseries 244} (2007), 668--690.

 \bibitem{dopa} (0862189)
    \newblock R.G.\,Douglas and V.I.\,Paulsen, 
  \newblock \emph{Completely bounded maps and hypo-Dirichlet algebras,} 
   \newblock Acta Sci. Math. (Szeged), {\bfseries 50} (1986), 143--157.

\bibitem{DrMc} (2163865)
 \newblock M.A.\,Dritschel and S.\,McCullough, 
 \newblock \emph{The failure of rational dilation on a triply connected domain, }
  \newblock J. Amer. Math. Soc. {\bfseries 18} (2005), no. 4, 873--918.

\bibitem{JaPf} (1242120)
 \newblock M.\,Jarnicki, P.\,Pflug, 
  \newblock ``Invariant Distances and Metrics in Complex Analysis'', 
  \newblock de Gruyter Expositions in Mathematics, 9. Walter de Gruyter, Berlin, 1993.


 \bibitem{lewis} (1086549)
  \newblock K.A.\,Lewis,
 \newblock \emph{Intersections of $K$-spectral sets, }
 \newblock J. Operator Theory {\bfseries 24} (1990), 129--135.


\bibitem{misra} (0749164)
 \newblock G.\,Misra, 
  \newblock \emph{Curvature inequalities and extremal properties of bundle shifts, } 
 \newblock J. Operator Theory {\bfseries 11}  (1984), 305--317.

\bibitem{Neu} (0043386)
 \newblock J.\,von Neumann, 
  \newblock \emph{Eine Spektraltheorie f\"ur allgemeine Operatoren eines unit\"aren Raumes,} 
  \newblock Math. Nachr. {\bfseries 4} (1951), 258--281.

\bibitem{paulsen}(0958575)
\newblock V.I.\,Paulsen,
  \newblock \emph{ Toward a theory of $K$-spectral sets}, 
   \newblock in : ``Surveys of Some Recent Results in Operator Theory'', Vol. I, 221--240, Pitman Res.
Notes Math. Ser., 171, Longman Sci. Tech., Harlow, 1988.

\bibitem{paul} (1976867)
  \newblock V.I.\,Paulsen,
  \newblock ``Completely Bounded Maps and Operator Algebras, ''
  \newblock Cambridge Studies in Advanced Mathematics, 78. Cambridge University Press,
Cambridge, 2002.

\bibitem{PauSin} (2285252)
\newblock V.I.\,Paulsen and D.\, Singh,
\newblock \emph{ Extensions of Bohr's inequality},
\newblock Bull. London Math. Soc.  {\bfseries 38} (2006), 991--999.

\bibitem{Pic}
\newblock J.\,Pickering, 
 \newblock \emph{Counterexamples to rational dilation on symmetric multiply connected domains, }
\newblock Preprint 2007, \arXiv:0711.4080v1 .

\bibitem{RiNa}(1068530)
\newblock F.\,Riesz and B.\,Sz.-Nagy, 
\newblock ``Functional Analysis'', 
\newblock Dover, 1991 (Reprint).

\bibitem{shields} (0361899)
\newblock A.L.\, Shields,
{\it Weighted shift operators and analytic function theory}, in :
\newblock Topics in operator theory, pp. 49--128. Math. Surveys, No. 13,
Amer. Math. Soc., Providence, R.I., 1974.

\bibitem{stampfliI} (0860344)
\newblock J.G. Stampfli, 
 \newblock \emph{Surgery on spectral sets, }
 \newblock J. Operator Theory, {\bfseries 16} (1986), no 2, 235--243.

\bibitem{stampfliII} (1047778)
\newblock J.G. Stampfli, 
\newblock \emph{Surgery on spectral sets. II. The multiply connected case, }
\newblock Integral Equations Operator Theory  {\bfseries 13}  (1990), 421--432.

\bibitem{simha}(0379831)
\newblock R.R.\ Simha, 
\newblock \emph{The Carath\'eodory metric of the annulus, }
\newblock Proc. Amer. Math. Soc.{\bfseries 50}  (1975),
162-166.

\bibitem{williams} (0205044)
\newblock J.P.\, Williams,
 \newblock \emph{Minimal spectral sets of compact operators, }
  \newblock Acta Sci. Math. (Szeged) {\bfseries 28} (1967) 93--106.

\end{thebibliography}

\end{document}